\documentclass{amsart}

\usepackage{tikz}
\usepackage{graphicx}	
\usepackage{amssymb}
\usepackage{amsmath}
\usepackage{amsthm, amsfonts, mathrsfs}

\newtheorem{theorem}{Theorem}
\newtheorem{lemma}[theorem]{Lemma}
\newtheorem{cor}[theorem]{Corollary}

\newtheorem{ques}{Question}

\newtheorem{exam}[theorem]{Example}

\title{On Vietoris-Rips complexes of Finite Metric Spaces with Scale $2$}
\date{January 2023}
\author[Z. Feng]{Ziqin Feng}
\address{Department of Mathematics and Statistics\\Auburn University\\Auburn, AL~36849}
\email{zzf0006@auburn.edu}

\author[P. Nukala]{Naga Chandra Padmini Nukala}
\address{Department of Mathematics and Statistics\\Auburn University\\Auburn, AL~36849}
\email{nzn0025@auburn.edu}
\date{\today}


\begin{document}
\begin{abstract} We examine the homotopy types of Vietoris-Rips complexes on certain finite metric spaces at scale $2$. We consider the collections of subsets of $[m]=\{1, 2, \ldots, m\}$ equipped with symmetric difference metric $d$, specifically, $\mathcal{F}^m_n$, $\mathcal{F}_n^m\cup \mathcal{F}^m_{n+1}$, $\mathcal{F}_n^m\cup \mathcal{F}^m_{n+2}$, and $\mathcal{F}_{\preceq A}^m$. Here $\mathcal{F}^m_n$ is the collection of size $n$ subsets of $[m]$  and $\mathcal{F}_{\preceq A}^m$ is the collection of subsets $\preceq A$ where $\preceq$ is a total order on the collections of subsets of $[m]$  and $A\subseteq [m]$ (see the definition of $\preceq$ in Section~\ref{Intro}). We prove that the Vietoris-Rips complexes $\mathcal{VR}(\mathcal{F}^m_n, 2)$ and $\mathcal{VR}(\mathcal{F}_n^m\cup \mathcal{F}^m_{n+1}, 2)$ are either contractible or homotopy equivalent to a wedge sum of $S^2$'s; also,  the complexes $\mathcal{VR}(\mathcal{F}_n^m\cup \mathcal{F}^m_{n+2}, 2)$ and  $\mathcal{VR}(\mathcal{F}_{\preceq A}^m, 2)$ are either contractible or homotopy equivalent to a wedge sum of $S^3$'s. We provide inductive formulae for these homotopy types extending the result of Barmak in \cite{Bar13} about the independence complexes of Kneser graphs \text{KG}$_{2, k}$ and the result of Adamaszek and Adams in \cite{AA22} about Vietoris-Rips complexes of  hypercube graphs with scale $2$.

\end{abstract}
\maketitle

\section{Introduction}\label{Intro}

Along with the development of topological data analysis \cite{RG08, GC09}, determining the homotopy types of Vietoris-Rips complex of finite metric spaces has become crucial in applied topology. In fact, the idea behind persistent homology is to compute the (co)homology of a Vietoris-Rips complex filtration built on data, which is typically a finite metric space in high dimensions (\cite{UB21}).
 Vietoris-Rips complexes were introduced by Vietoris in \cite{VI27} and then by Rips (see \cite{EG87}) to approximate a metric space at a chosen scale for different purposes. Additionally, these kinds of complexes have been intensively used in computational topology as a simplical model for the sensor networks (\cite{GM05, SG06, SG07}) and as a tool for image processing (\cite{SMC07}).

The Vietoris-Rips complex $\mathcal{VR}(X;r)$ of a metric space $(X,d)$ with scale $r\geq 0$ is a simplicial complex with vertex set $X$, where a nonempty subset $\sigma\in [X]^{<\infty}$ is a simplex in $\mathcal{VR}(X;r)$ if and only if its diameter satisfies $\text{diam}(\sigma)\leq r$. Here, $[X]^{<\infty}$ denotes the collection of all finite subsets of $X$, and for any subset $S$ of $X$ $\text{diam}(S)$ is defined as the supremum of all distances $d(x,y)$ between pairs of points $x,y\in S$. Recent work has focused on studying Vietoris-Rips complexes of circles (\cite{AA17}), metric graphs (\cite{GGPSWWZ18}), geodesic spaces (\cite{ZV19, ZV20}), and more.

In this paper, we investigate the homotopy type of the Vietoris-Rips complex $\mathcal{VR}(\mathcal{F}, 2)$ of a specific class of finite metric spaces with scale $2$. Let $\mathcal{F}$ be a collection of subsets of $[m]$ for some $m\in \mathbb{N}$, where $[m]=\{1,2,\ldots,m\}$. We define a metric $d$ on $\mathcal{F}$ such that, for any $A$ and $B$ in $\mathcal{F}$, $d(A, B)=|A\Delta B|$, where $A\Delta B$ denotes the symmetric difference of $A$ and $B$, i.e., $(A\setminus B)\cup (B\setminus A)$. Hence, $(\mathcal{F}, d)$ is a finite metric space. In this paper, we study the homotopy types of the Vietoris-Rips compelxes, $\mathcal{VR}(\mathcal{F}_n^m, 2)$ (Section~\ref{F_n_m}), $\mathcal{VR}(\mathcal{F}_{\preceq A}^m, 2)$ (Section~\ref{F_A_m}), and $\mathcal{VR}(\mathcal{F}_{n}^m\cup \mathcal{F}_{n'}^m, 2)$ (Section~\ref{F_p_q_m}), where $\mathcal{F}_n^m$, $\mathcal{F}_{\preceq A}^m$, and $\mathcal{F}_{n}^m\cup \mathcal{F}_{n'}^m$ are all collections of subsets of $[m]$.  We show that:

\begin{itemize}
\item[i)] the complexes $\mathcal{VR}(\mathcal{F}_n^m, 2)$ and $\mathcal{VR}(\mathcal{F}_{n}^m\cup \mathcal{F}_{n+1}^m, 2)$ are either contractible or homotopy equivalent to a wedge sum of $S^2$'s;

\item[ii)] the complexes $\mathcal{VR}(\mathcal{F}_{\preceq A}^m, 2)$  and $\mathcal{VR}(\mathcal{F}_{n}^m\cup \mathcal{F}_{n+2}^m, 2)$ are either contractible or homotopy equivalent to a wedge sum of $S^3$'s.
\end{itemize}
Furthermore, we identify inductive formulas for determining the homotopy types of these complexes. The homotopy type of $\mathcal{VR}(\mathcal{F}, r)$ for $r\geq 0$ is closely related to the study of the independence complex of Kneser graphs in \cite{Bar13} and the Vietoris-Rips complexes of hypercube graphs in \cite{AA22}.

The independence complex I$(G)$ of a graph $G=(V(G), E(G))$ is a simplicial complex whose simplices are the independent sets of vertices of $G$, i.e., sets of vertices no two of which are adjacent. The Kneser graph $\text{KG}_{n, k}$ has the $n$-subsets of $[2n+k]$ as its vertices  and its edges are given by pairs of disjoint such subsets. In particular, any two vertices in $\text{KG}_{n, k}$ are not disjoint if and only if their symmetric difference distance is at most $2n-1$. Therefore, the independence complex of $\text{KG}_{n, k}$ is identical to the Vietoris-Rips complex $\mathcal{VR}(\mathcal{F}_n^{2n+k}, 2n-1)$, where $\mathcal{F}_n^{m}$ denotes the collection of all $n$-subsets of $[m]$.

Barmak proved in \cite{Bar13} (Theorem 4.11) that the independence complex of $\text{KG}_{2, k}$, I($\text{KG}_{2, k}$),  is homotopy equivalent to $\bigvee_{k+3\choose 3} S^2$. For any $m\geq 4$, note that $\mathcal{VR}(\mathcal{F}_2^m, 2)=\mathcal{VR}(\mathcal{F}_2^m, 3)=\text{I}(\text{KG}_{2, m-4})$; thus, the  complex $\mathcal{VR}(\mathcal{F}_{2}^{m}, 2)$ is homotopy equivalent to a wedge sum of ${m-1\choose 3}$ copies of $S^2$. Our result on the homotopy types of $\mathcal{VR}(\mathcal{F}_{n}^{m}, 2)$ (Corollary~\ref{homot_n_2}) is a generalization of Barmak's result.  When $m=2n$, the complex $\mathcal{VR}(\mathcal{F}_n^{m}, m-2)$ has ${m\choose n}$ vertices and is the boundary of a cross-polytope, so it is homotopy equivalent to $S^{\frac{1}{2}{m\choose n}-1}$.

The hypercube graph is a graph whose vertices are all binary strings of length $m$, denoted by $Q_m$, and whose edges are given by pairs of such strings with Hamming distance $1$. The Hamming distance between any two binary strings with the same length is defined as the number of positions in which their entries differ. We can consider $Q_m$ as a metric space equipped with the Hamming distance, and then the hypercube graph can be identified as the complex $\mathcal{VR}(Q_m, 1)$.

Adamaszek and Adams investigated the Vietoris-Rips complexes $\mathcal{VR}(Q_m, r)$ at small scales $r=0,1,2$ in their recent work \cite{AA22}. The complex $\mathcal{VR}(Q_m, 0)$ is homotopy equivalent to a wedge sum of $(2^m-1)$-many $S^0$'s, and $\mathcal{VR}(Q_m, 1)$ is homotopy equivalent to a wedge sum of $((m-2)2^{m-1}+1)$-many $S^1$'s. Their main result is that the complex $\mathcal{VR}(Q_m, 2)$ is homotopy equivalent to a wedge sum of $c_m$ copies of $S^3$'s, where $c_m$ is given by $c_m=\sum_{0\leq j<i<m}(j+1)(2^{m-2}-2^{i-1})$. The \v{C}ech complexes of the metric space $Q_m$ with scales $2$ and $3$ are studied in \cite{ASS22}.

Each binary string of length $m$ can also be considered as the characteristic function of a subset of $[m]$. Hence, there is a natural isometric map between the metric spaces $Q_m$ and $\mathcal{P}([m])$, where $\mathcal{P}([m])$ is the collection of all subsets of $[m]$ equipped with the symmetric difference metric $d$. Notice that $\mathcal{P}([m])$ contains the empty set $\emptyset$ as an element. Hence the result about the homotopy type of $\mathcal{VR}(Q_m, 2)$ by Adamaszek and Adams is a special case of Theorem~\ref{allpowersets} which gives a deeper understanding on how its homotopy type is formed.    Adamaszek and Adams in \cite{AA22} used Polymake \cite{Poly10} and Ripser++ \cite{Zhang} to compute the reduced homology groups of $\mathcal{VR}(\mathcal{P}[m], 3)$ for $m=5, 6, \ldots, 9$, with coefficients $\mathbb{Z}$ or $\mathbb{Z}/2\mathbb{Z}$. They found that these homology groups are nontrivial only in dimensions $4$ and $7$, indicating that the complex $\mathcal{VR}(\mathcal{P}[m], 3)$ is a wedge sum of copies of $S^4$'s and $S^7$'s. This suggests that the homotopy type of the complex $\mathcal{VR}(\mathcal{P}[m], 3)$ is more complicated than that of the complexes $\mathcal{VR}(\mathcal{P}[m], r)$ with $r=0, 1, 2$. Shukla \cite{Shu22} subsequently proved that for $m\geq 5$, the reduced homology group $\tilde{H}_i(\mathcal{VR}(\mathcal{P}([m]), 3))$ is nontrivial if and only if $i\in \{4, 7\}$.

\medskip

In this paper, we extend the study of Vietoris-Rips complexes to other collections of subsets in $[m]$ with scale $2$ beyond $\mathcal{F}_{2}^m$ and $\mathcal{P}[m]$.  To determine the homotopy type of $\mathcal{VF}(\mathcal{P}[m], 2)$, Adamaszek and Adams in \cite{AA22} used an inductive proof on the clique complex of the graph $G_\ell^2$, whose vertices are binary sequences of non-negative integers $\leq \ell-1$ with edges given by pairs of sequences with Hamming distance $\leq 2$. We adopt a different inductive process to study these complexes and our approach is also potentially applicable to the investigation of these complexes at larger scales.

We start with introducing notations for certain collections of subsets of $[m]$. For $n\leq m$,  let $\mathcal{F}_{\leq n}^m$ be the collection of all subsets of $[m]$ with cardinality $\leq n$. It is easy to see that the complex $\mathcal{VR}(\mathcal{F}^m_{\leq r}, r)$ is contractible since it is a cone with the cone vertex being the empty set $\emptyset$.  We now proceed to define a total ordering $\prec$ on $\mathcal{P}([m])$ to facilitate the conduction of induction process.  For each $A\subseteq [m]$ with $|A|=n$, we represent $A=\{i_1, i_2, \ldots, i_n\}$ as $i_1i_2\cdots i_n$ with $i_1<i_2<\cdots<i_n$. For any $A, B\subseteq  [m]$,  we say $A\prec B$ if one of the followings holds:
\begin{itemize}
\item[i)] $|A|<|B|$;

\item[ii)] there is a $k\in\mathbb{N}$ such that $i_k<j_k$ and $i_{\ell}=j_{\ell}$ for any $\ell<k$, when $n=|A|=|B|$, $A=i_1i_2\cdots i_n$ and $B=j_1j_2\cdots j_n$.

\end{itemize}
Clearly this is a total order on $\mathcal{P}([m])$ and for any subcollection $\mathcal{F}$ of $\mathcal{P}([m])$, $(\mathcal{F}, \prec)$ is also a total order. For any $A\subset [m]$, we denote $\mathcal{F}^m_{\prec A}=\{B: B\prec A \text{ and } B\subset [m]\}$ and $\mathcal{F}^m_{\preceq A}=\mathcal{F}^m_{\prec A}\cup \{A\}$. Notice that the set $[m]$ is the maximal elements in $\mathcal{P}([m])$; hence if $A=[m]$, $\mathcal{F}_{\preceq A}^m =\mathcal{P}([m])$.

We start with some easy observations of the homotopy types of such complexes. For any collection $\mathcal{F}$ of subsets of $[m]$, $\mathcal{VR}(\mathcal{F}, 0)$ is  a complex with $|\mathcal{F}|$-many disjoint vertices. Also for any $1\leq n\leq m-1$, $\mathcal{VR}(\mathcal{F}_n^m, 1)$ is also the space of ${m\choose n}$ disjoint vertices since $d(A, B)\geq 2$ for any two different subsets $A, B$ with cardinality $n$. Also for each $i=0, 1, \ldots, m$, the metric space $\mathcal{F}_i^m$ is isometric to $\mathcal{F}_{m-i}^m$ since the complementary mapping with $\phi(A)=[m]\setminus A$ preserves the symmetric distance from $\mathcal{F}_i^m$ to $\mathcal{F}_{m-i}^m$. Therefore $\mathcal{VR}(\mathcal{F}_i^m, r)$ is homotopy equivalent to $\mathcal{VR}(\mathcal{F}_{m-i}^m, r)$ for each $r\geq 0$.   We see that the complexes $\mathcal{VR}(\mathcal{F}^m_1, 2)$ and $\mathcal{VR}(\mathcal{F}^m_{m-1}, 2)$ are contractible because each pair of their vertices has distance $2$. Hence the complex $\mathcal{VR}(\mathcal{F}^m_n, 2)$ is contractible when $n=0, 1, m-1,$ or $m$. Similarly the complexes $\mathcal{VR}(\mathcal{F}_{n}^m\cup \mathcal{F}_{n+1}^m, 2)$ is contractible if $n=0$ or $ m-1$.

\section{Notations and Preliminary Results}\label{prel}

\textbf{Topological Spaces and Wedge sums.} Let $X$ and $Y$ be topological spaces. We write $X\simeq Y$ when they are homotopy equivalent. We denote $S^k$ to be the $k$-dimensional sphere. The wedge sum of $X$ and $Y$, $X\vee Y$, is the space obtained by gluing $X$ and $Y$ together at a single point. The homotopy type of $X\vee Y$ is independent of the choice of points if $X$ and $Y$ are connected CW complexes. For $k\ge 1$, $\bigvee_k X$ denotes the $k$-fold wedge sum of $X$.  We denote $\Sigma X$ to be the suspension of $X$. For any sphere $S^k$, $\Sigma S^k$ is homeomorphic to $S^{k+1}$. A function $f$ from $X$ to $Y$ is said to be null-homotopic if it is homotopic to a constant map. It is well-known that any mapping from $S^n$ to $S^m$ is null-homotopic when $n<m$.

Any two metric spaces $(X, d_X)$ and $(Y, d_Y)$ are said to be isometric if there is a bijective distance-preserving map $f$ from $X$ to $Y$, i.e., $d_X(x_1, x_2)=d_Y(f(x_1), f(x_2))$ for any $x_1, x_2\in X$. Hence if $X$ and $Y$ are isometric, then it is straightforward to verify that  $\mathcal{VR}(X, r)$ is homeomorphic to $\mathcal{VR}(Y, r)$ for any $r\geq 0$.

A cross-polytope with $2n$ vertices is a regular, convex polytope that exists in $n$-dimensional Euclidean space. So it homeomorphic to the unit ball in $\mathbb{R}^n$ whose boundary is homeomorphic to $S^{n-1}$.

\medskip

\textbf{Simplicial complexes.} A simplicial complex $K$ on a vertex set $V$ is a collection of non-empty subsets of $V$ such that: i) all singletons are in $K$; and ii) if $\sigma\in K$ and $\tau\subset \sigma$, then $\tau\in K$. For a complex $K$, we use $K^{(k)}$ to represent the $k$-skeleton of $K$, which is a subcomplex of $K$. For vertices $v_1, v_2, \ldots, v_k$ in a complex $K$,  if they span a simplex in $K$, then we denote the simplex to be $\{v_1, v_2, \ldots, v_k\}$. If $\sigma$ and $\tau$ are simplices in $K$ with $\sigma\subset \tau$, we say $\sigma$ is a face of $\tau$.  We say a simplex is a maximal simplex (or a facet) if it is not a face of any other simplex. We say that $L$ is a full subcomplex of $K$ if it contains all the simplicies in $K$ spanned by the vertices in $L$.

A complex $K$ is \emph{a clique complex} if  the following condition holds: a non-empty subset $\sigma$ of vertices is in $K$ given that the edge $\{v, w\}$ is in  $K$ for any pair $v, w\in \sigma$. For any graph $G=(V, E)$, we denote Cl$(G)$ to be the clique complex of $G$ whose vertex set is $V$ and Cl$(G)$ contains a finite subset $\sigma\subset V$ as a simplex if each pair of vertices in $\sigma$ forms an edge in $G$. Also, we see that the Vietoris-Rips complex over any metric space is a clique complex by the definition.

Let $L$ be a complex and $v$ be a vertex not in $L$. The cone over $L$ with the vertex $v$, denoted by $v\ast L$, is a simplical complex defined on the vertex set $L^{(0)}\cup \{v\}$ such that a simplex of $v\ast L$ is either a simplex in $L$ or a simplex in $L$ adjoined with $v$. Notice that any cone is contractible.

\medskip

The following result is proved in \cite{GSS22}. This is an important method to investigate the homotopy type of a complex by splitting it into two or more subcomplexes.

\begin{lemma}\label{cup_simp} Let $K$ be a simplical complex. Suppose that  $K=K_1\cup K_2$ and the inclusion maps $\imath_1: K_1\cap K_2\rightarrow K_1$ and $\imath_2: K_1\cap K_2\rightarrow K_2$ are both null-homotopic. Then $$K\simeq K_1\vee K_2\vee\Sigma(K_1\cap K_2).$$
\end{lemma}

The next lemma (see \cite{AA22}, Lemma 1) is an easy corollary of this result. For any vertex $v$ in a complex $K$, $K\setminus v$ denote the induced complex on the vertex set $K^{(0)}\setminus \{v\}$. The star of a vertex $v$ in $K$ is st$_K(v)=\{\sigma: \sigma\cup \{v\}\in K\}$. Hence for any $v\in V$, st$_K(v)$ is contractible because it is a cone over $\text{lk}_K(v)$ with the vertex $v$, namely $v\ast \text{lk}_K(v)$, where $\text{lk}_K(v)=\{\sigma: \sigma\cup \{v\}\in K\text{ and }v\notin \sigma\}$.

\begin{lemma} \label{complex_add_1v}If $v$ is a vertex in $K$ with the inclusion map $\imath: \text{lk}_K(v)\rightarrow K$ being null-homotopic, then $K$ is homotopic to $K\setminus v\vee \Sigma (\text{lk}_K(v))$.

\end{lemma}
\begin{proof} Let $v$ be a vertex in $K$ such that the inclusion map $\imath: \text{lk}_K(v)\rightarrow K$ is null-homotopic. It is straightforward to verify that $K=\text{st}_K(v)\cup K\setminus v$. Since $\text{st}_K(v)$ is contractible, the inclusion from $\text{lk}_K(v)$ to $\text{st}_K(v)$ is also null-homotopic. Notice that $(K\setminus v)\cap \text{st}_K(v) = \text{lk}_K(v)$. Hence by Lemma~\ref{cup_simp} and the fact that $\text{st}_K(v)$ is contractible,
$$K\simeq K\setminus v\vee \text{st}_K(v)\vee \Sigma(\text{lk}_K(v))\simeq  K\setminus v\vee \Sigma(\text{lk}_K(v)).   $$ \end{proof}

\begin{lemma}\label{subcom_simplex} If $\sigma$ is a $k$-simplex and $K_\sigma$ is the simplicial complex generated by $\sigma$, then $K_\sigma^{(n)}$ is homotopy equivalent to a wedge sum of $k \choose n+1$-many of $S^{n}$ for any $n<k$.
\end{lemma}

\begin{proof} Assume $\sigma=\{v_0, v_1, \ldots, v_k\}$ and $K_\sigma$ is the simplicial complex generated by $\sigma$. Denote $K=K_\sigma^{(n)}$.  Notice that there are ${k\choose n+1}$-many $n$-simplices which don't contain $v_0$.  List such $n$-simplices as $\tau_1, \tau_2, \ldots, \tau_{k\choose n+1}$. Then $K=\text{st}_K(v_0)\cup \bigcup\{K_{\tau_i}: i=1, 2, \ldots, {k\choose n+1}\}$ where $K_{\tau_i}$ is the simplicial complex generated by $\tau_i$ for each $i$. Next we show that $K\simeq \bigvee_{k\choose n+1}S^n$ by induction.

First notice that $\text{st}_K(v_0)\cap K_{\tau_1}$ is the boundary of $\tau_1$, and hence is homotopy equivalent to $S^{n-1}$. Since both $\text{st}_K(v_0)$ and $K_{\tau_1}$ are contractible, the inclusion maps from their intersection to each of them are null-homotopic. By Lemma~\ref{complex_add_1v}, $\text{st}_K(v_0)\cup K_{\tau_1}\simeq \Sigma S^{n-1}\simeq S^n$. Then inductively, $K\simeq \bigvee_{k\choose n+1} S^n$.      \end{proof}

Also in this paper for convenience, we set $\sum_{i=a}^{b} f(i)=0$ when $b<a$, where $f$ is a function on the set of natural numbers.

\section{Star Cluster of a subcomplex}

To investigate the topology of the independence complex of graphs, Barmak \cite{Bar13} introduced a general tool using which he answered a question arisen from works of Engstr\"{o}m and Jonsson and investigated lots of examples appearing from literature. It turns out this concept is a powerful tool to understand general simplicial complexes.  For any subcomplex $L$ of $K$, we define the \emph{star cluster} of $L$ in $K$ as the subcomplex  $$\text{SC}_K(L)=\bigcup_{v\in L^{(0)}} \text{st}_K(v).$$

If $\sigma$ is a simplex in $K$, Barmak in \cite{Bar13} proved that  $\text{SC}_K(\sigma)$ is contractible, hence homotopy equivalent to $\sigma$. In general, given that $L$ is a subcomplex of $K$, $\text{SC}_K(L)$ is not homotopy equivalent to $L$ as showed in the example below.

\begin{exam} Let $K=\mathcal{VR}(\mathcal{P}([2]), 1)$ and $L$ be the full subcomplex with vertices $\{\emptyset, \{1\}, \{2\}\}$. Then $L$ is contractible but in the other hand,  SC$_K(L)=K$ which is homotopy equivalent to $S^1$. \end{exam}

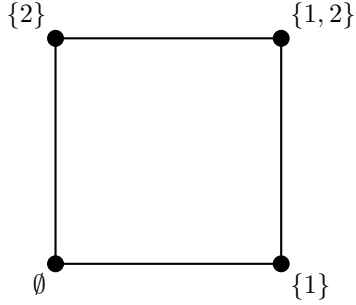
\begin{center}
\begin{tikzpicture}[scale=3]
\draw[thick] (0,0) -- (1,0) -- (1,1) -- (0,1) -- cycle;
\filldraw[black] (0,0) circle (1pt) node[below left]{$\emptyset$};
\filldraw[black] (1,0) circle (1pt) node[below right]{$\{1\}$};
\filldraw[black] (1,1) circle (1pt) node[above right]{$\{1, 2\}$};
\filldraw[black] (0,1) circle (1pt) node[above left]{$\{2\}$};
\end{tikzpicture}
\end{center}

Next, we'll give a sufficient condition under which the star cluster of a subcomplex $L$ in $K$ is homotopy equivalent to $L$. This result is a generalization of Barmak's result about SC$_K(\sigma)$ being contractible for any simplex $\sigma$ in $K$; and it is also  heavily used to determine the homotopy type of simplicial complexes in this paper.

\begin{lemma}\label{SC_homo} Let $K$ be a clique complex and $L$ a clique subcomplex of $K$. Suppose that the edge $\{v, w\}$ is in $L$ for any pair $v, w\in L^{(0)}$ with $(\text{st}_K(v)\cap \text{st}_K(w))\setminus L\neq \emptyset$. Then the following hold:
%\begin{itemize}
%\item[i)]if $\sigma$ is a simplex in $K$ with vertices in $L$, then $\sigma\in L$;

%\item[] for any $v, w\in L$, $\{v, w\}\in L$ if $(\text{St}_K(v)\cap \text{St}_K(w))\setminus L\neq \emptyset$.
%\end{itemize}

\begin{itemize}

\item[i)] $L$ is a full subcomplex of $K$; % a simplex $\sigma$ is in $L$ if all its vertices are in $L$;

\item[ii)] for any collection of vertices, $v_1, v_2, \ldots, v_\ell$ in $L$, the complex  $L'=L\cup \bigcup_{i=1}^\ell\text{st}_K(v_i))$ is homotopy equivalent to $L$.  % \[(L\cup \bigcup_{i=0}^n\text{st}_K(v_i))\simeq L. \]

\end{itemize}
In particular, ii) implies that $\text{SC}_K(L)$ is homotopy equivalent to $L$.
\end{lemma}

\begin{proof} First we prove i). Let $\sigma=\{w_0, w_1, \ldots, w_k\}$ be a simplex in $K$ and $w_j\in L$ for each $j=0, 1, \ldots, k$. Take an arbitrary pair $w_j, w_{j'}$ of vertices from $\sigma$ with $j\neq j'$. Suppose, for contradiction, that $\{w_j, w_{j'}\}\notin L$.  Since the $1$-simplex $\{w_j, w_{j'}\}$ is in $K$, it is in both  $\text{st}_K(w_j)$ and $\text{st}_K(w_{j'})$. Hence, $\text{st}_K(w_j)\cap \text{st}_K(w_{j'})\setminus L\neq \emptyset$. Then by the assumption the edge $\{w_j, w_{j'}\}\in L$ which is a contradiction. Therefore each pair of vertices in $\sigma$ forms an edge in $L$. Since $L$ is a clique complex, $\sigma\in L$.

We'll prove ii) by induction. Suppose that the vertices $v_1, v_2, \ldots, v_{k-1}$ in $L$  satisfy that the complex $L_0= L\cup \bigcup\{\text{st}_{K}(v_i): i=1, 2, \ldots, k-1\}\simeq L$. When $k=1$, $L_0=L$ and the result holds. Let $v_k$ be any other vertex in $L$ and $L_1=L_0\cup \text{st}_K(v_k)$. We'll show that $L_1\simeq L$.

We claim that $L_0\cap \text{st}_{K}(v_k)=\text{st}_{L_0}(v_k)$. Note that both $\text{st}_{K}(v_k)$ and $\text{st}_{L_0}(v_k)$ are contractible, hence so is $\Sigma(\text{st}_{L_0}(v_k))$.  Then by Lemma~\ref{cup_simp} and the inductive assumption, $$L_1= L_0\cup \text{st}_{K}(v_k)\simeq L_0\vee \Sigma (\text{st}_{L_0}(v_k))\vee \text{st}_{K}(v_k)\simeq L_0\simeq L.$$

Next we prove our claim above. The inclusion $\text{st}_{L_0}(v_k)\subseteq L_0\cap \text{st}_{K}(v_k)$ is clear from definition. Then, take a simplex $\sigma\in L_0\cap \text{st}_{K}(v_k)$ and we'll prove $\sigma\in \text{st}_{L_0}(v_k)$ in the following two cases.
\begin{itemize}
\item[Case a):] Suppose that all the vertices of $\sigma$ are in $L$. Since $\sigma\in \text{st}_{K}(v_k)$, $\sigma\cup\{v_k\}$ is a simplex in $K$ whose vertices are in $L$. Then by i), $\sigma\cup\{v_k\}\in L\subseteq L_0$; hence $\sigma\in \text{st}_{L_0}(v_k)$.

\item[Case b):] Suppose that the simplex $\sigma$ contains at least one vertex  not in $L$. Then clearly $\sigma\notin L$. Because  $\sigma\in L_0$, then there exists at lease one $k_0$ with $1\leq k_0 \leq k-1$ such that $\sigma\in \text{st}_K(v_{k_0})$. So $\sigma\cup \{v_{k_0}\}$ is a simplex in $K$.   Since $\sigma \in \text{st}_{K}(v_k)$, $\sigma\cup \{v_{k}\}$ is also a simplex in $K$. Also note that $\sigma\in (\text{st}_K(v_{k_0})\cup \text{st}_{K}(v_k))\setminus L$. By the assumption $\{v_{k_0}, v_k\}$ is an edge in $K$. Since $K$ is a  clique complex, $\sigma\cup \{v_{k_0}, v_k\}$ is a simplex in $K$; and this simplex is in  $\text{st}_{K}(v_{k_0})\subseteq L_0$. Hence the simplex $\sigma\cup \{v_{k_0}, v_k\}$ is in $\text{st}_{L_0}(v_k)$ which implies that $\sigma$ is also in $\text{st}_{L_0}(v_k)$.

\end{itemize}
\end{proof}
%Then all the vertices of $\sigma$ is in $L_0$ and  $\sigma\cup \{v_k\}$ is a simplex in $K$. Therefore by the inductive assumption, $\sigma\cup \{v_k\}$ is a simplex in $L_0$. So, $\sigma\in \text{st}_{L_0}(v_k)$. Hence $L_0\cap \text{st}_{K}(v_k)=\text{st}_{L_0}(v_k)$.

%Lastly, we show that the complex $L'$ satisfies condition i). Pick a simplex $\sigma$ in $K$ with vertices being in $L'$.  If all the vertices of $\sigma$ are in $L$, then by i) $\sigma\in L\subseteq L_0$. If none of the vertices in $\sigma$ is in $L$, then $\sigma^{(0)}\subset \text{st}_K(v_0)$ which means that for any $w\in \sigma$, $[w, v_0]\in K$. Therefore in this case, $\sigma\cup \{v_0\}\in K$ which implies that $\sigma\in \text{st}_K(v_0)\subseteq L_0$. Now suppose that $\sigma^{(0)}\cap L\neq \emptyset$ and  $\sigma^{(0)}\setminus L\neq \emptyset$. Clearly for any $w\in \sigma^{(0)}\setminus L $, $[w, v_0]\in K$. Fix $w_0\in \sigma^{(0)}\setminus L$, clearly $w_0\in \text{st}_K(v_0)$. Take $v\in \sigma^{(0)}\cap L$. Since $[w_0, v]\in K$, $w_0\in \text{st}_K(v)$. Therefore, $(\text{st}_K(v_0)\cap\text{st}_K(v))\setminus L\neq \emptyset$, hence by the assumption $\{v_0, v\}\in L\subset K$. Therefore $\sigma\cup \{v_0\}\in K$ and $\sigma\in \text{st}_K(v_0)\subseteq L_0$.

%\end{proof}

Next, we give a way to split a complex $K$ into a union of two subcomplexes using star clustering. Then we could apply Lemma~\ref{cup_simp} to investigate the homotopy type of the complex $K$.

\begin{lemma}\label{SC_complex} Let $K$ be a simplicial complex and $K_1, K_2$ be subcomplexes of $K$ such that
\begin{itemize}
\item[i)]$K^{(0)}=K_1^{(0)}\cup K_2^{(0)}$;

\item[ii)] $K_2$ is a full subcomplex of $K$. %$\sigma\in K_2$ if $\sigma$ is in $K$ and the vertices of $\sigma$ are in $K_2$.

\end{itemize}

Then $K=\text{SC}_{K}(K_1)\cup K_2$. \end{lemma}

\begin{proof} Let $\sigma$ be a simplex of $K$. If one of $\sigma$'s vertices, namely $v$, is in $K_1$, then $\sigma\in\text{st}_K(v) \subseteq \text{SC}_K(K_1)$; otherwise, $\sigma\in K_2$ by the assumption.
\end{proof}

\section{Vietoris-Rips Complex $\mathcal{VR}(\mathcal{F}_{n}^{m},2)$}\label{F_n_m}

Starting from this section, each vertex of a complex is a subset of $[m]$ and we'll use $A$, $B$, $C$, or $D$ to represent them.   For any subset $C$ of $[m]$, denote $N[C]=\{A\in \mathcal{P}([m]): C\subset A \text{ and }|A\setminus C|=1\}$ and  $L[C]=\{A\in \mathcal{P}([m]): A\subset C\text{ and }|C\setminus A|=1\}$.

% It is known (see \cite{Bar13}, Theorem 4.11) that the independence complex of KG$_{2, k}$ is homotopy equivalent to $\bigvee_{{k+3\choose 3}} S^2$. Note that the independence complex of KG$_{2, k}$ is identical with the Vietoris-Rips complex $\mathcal{VR}(\mathcal{F}_2^{k+4}, 2)$.

 Fix $n, m\in\mathbb{N}$ with $n<m$. For any $\{i_1,i_2,\ldots,i_n, i_{n+1}\}\in [m]$ with $i_1<i_2<\ldots<i_n<i_{n+1}$, we get that  $$N[i_1,i_2,\ldots ,i_{n-1}]=\{A\in \mathcal{F}_{n}^m:\{i_1,i_2,\ldots,i_{n-1}\}\subset A\}, \text{ and }$$
 $$L[i_1,i_2,\ldots,i_{n+1}]=\{ i_1i_2\cdots \hat{i_j}\cdots i_{n+1}: j\in \{1,...,n+1\} \},$$
here, $i_1i_2\cdots \hat{i_j}\cdots i_{n+1}$ is defined to be  $\{i_1,i_2,\ldots,i_n, i_{n+1}\}\setminus \{i_j\}$ for each $j$.

\begin{lemma} Assume that $m\geq n+2$ with $n\geq 2$ and $\{i_1, i_2, \ldots, i_{n+1}\}\subseteq [m]$.  Then $N[i_1,i_2,..,i_{n-1}]$ and $L[i_1,i_2,\ldots,i_{n+1}]$ are maximal simplices in  the complex $\mathcal{VR}(\mathcal{F}_n^{m}, 2)$.

\end{lemma}

\begin{proof} It is straightforward to verify that $N[i_1,i_2,\ldots,i_{n-1}]$ is an $(m-n)$-simplex  and $L[i_1,i_2,\ldots,i_{n+1}]$ is an $n$-simplex in $\mathcal{VR}(\mathcal{F}_n^{m}, 2)$.

First, we show that $N[i_1,i_2,\ldots,i_{n-1}]$ is a maximal simplex in $\mathcal{VR}(\mathcal{F}_n^{m}, 2)$.  Let $A$ be an $n$-subset of $[m]$ such that $A\notin N[i_1,i_2,\ldots,i_{n-1}]$. Without loss of generality, we assume that $i_1\notin A$, then we pick $i,j\in A\setminus \{i_1, i_2,\ldots, i_{n-1}\}$ and $k\in[m]\setminus \{i, j, i_1, i_2, \ldots, i_{n-1}\} $. Let $B=\{i_1, i_2, \ldots, i_{n-1}, k\}$ which is clearly in $N[i_1, i_2, \ldots, i_{n-1} ]$. Clearly, $d(A, B)\geq 4$ since $\{i, j, k, i_1\}\subseteq A\Delta B$.  Hence $N[i_1,i_2,..,i_{n-1}]$ is a maximal simplex in $\mathcal{VR}(\mathcal{F}_n^{m}, 2)$.

Next, we show that  $L[i_1,i_2,\ldots,i_{n+1}]$ is a maximal simplex in $\mathcal{VR}(\mathcal{F}_n^{m}, 2)$. Let $A$ be an $n$-subset of $[m]$ such that $A\notin L[i_1,i_2,\ldots,i_{n+1}]$. Then there is an $i\in [m]\setminus \{i_1,i_2,\ldots,i_{n+1}\}$ such that $i\in A$. Suppose $A\cap \{i_1, i_2, \ldots, i_{n+1}\}\neq \emptyset$. Then we can pick $B\in L[i_1,i_2,\ldots,i_{n+1}]$ such that $ (\{i_1,i_2,\ldots,i_{n+1}\}\setminus A) \subset B$. Then $|A\setminus B|\geq 1$ since $i\in A\setminus B$. Also, $|B\setminus A|\geq 2$ since at least $2$ elements in $\{i_1, i_2, \ldots, i_{n+1}\}$ are not in $A$. So $d(A, B)\geq 3$. If $A\cap \{i_1, i_2, \ldots, i_{n+1}\}= \emptyset$, $d(A, B)\geq 2n\geq 4$ for any $B\in L[i_1,i_2,\ldots,i_{n+1}]$. Hence $L[i_1,i_2,\ldots,i_{n+1}]$ is a maximal simplex in $\mathcal{VR}(\mathcal{F}_n^{m}, 2)$.     \end{proof}

For convenience in this paper, we will use $N[i_1,i_2,\ldots ,i_{n-1}]$ or  $L[i_1,i_2,\ldots,i_{n+1}]$ to represent both a simplex and the subcomplex generated by the simplex in $\mathcal{VR}(\mathcal{F}_n^m, 2)$ or any other complexes containing them.

For a complex $K$, let $M(K)$ be the collection of maximal simplices in $K$. Clearly $K=\bigcup M(K)$. Hence it is important to understand the collection of maximal simplices in a complex. Next, we show that there are only these two of maximal simplices in $\mathcal{VR}(\mathcal{F}_n^{m}, 2)$.

\begin{lemma} \label{max_simp}
Fix $n, m\in\mathbb{N}$ with $1<n<m$. Let $K$ be the complex $\mathcal{VR}(\mathcal{F}_n^{m}, 2)$.

\begin{itemize}

\item[i)] Any maximal simplex $\sigma$ in $K$ is  either $N[i_1,i_2,..,i_{n-1}]$  or  $L[i_1,i_2,..,i_{n+1}]$ for $i_1,i_2,i_3,..., i_{n+1} \in [m]$ with $i_1<i_2<...<i_n<i_{n+1}$.
\item[ii)] For any $k\geq 2$ and $\{A_1, A_2, \ldots, A_{k+1}\}$ being a $k$-simplex in $K$ such that $|\bigcap_{\ell=1}^{k+1}A_\ell|<n-1$, the only maximal simplex containing $\{A_1, A_2, \ldots, A_{k+1}\}$ as a face is $L[A_1\cup A_2]$.

\end{itemize}

\end{lemma}

\begin{proof}
To prove i), we pick a maximal simplex  $\sigma$  in the complex $K$.

Note that the vertices of $K$ are subsets of $[m]$. Hence, $\sigma$ is a collection of subsets of $[m]$ and $\bigcap \sigma\subset [m]$. If $|\bigcap \sigma|=n-1$, then clearly $\sigma$ is one of the simplices in the form $N[i_1,i_2,..,i_{n-1}]$.

We claim that the size of the set $\bigcap \sigma$ can't be greater than $0$ and less than $n-1$.  For the purpose of contradiction, we suppose that $0<|\bigcap \sigma|<n-1$. Let $|\bigcap \sigma|=k$ with $0<k<n-1$ and list $\bigcap \sigma$ as $\{i_1, i_2, \cdots,  i_k\}$. Pick $A\in \sigma$ such that $A\setminus \bigcap \sigma =\{j_{1}, j_2, \ldots, j_{n-k}\}$.  For each $\ell=1, 2, \ldots, n-k$,   pick $B_\ell\in \sigma$ such that $j_\ell\notin B_\ell$. Also $|B_\ell\setminus A|=1$ because $d(B_\ell, A)=2$ for each $\ell$. Since $k<n-1$, $n-k\geq 2$. Then we let $i_0$ be the number in $B_1\setminus A$ and $j_0$ be the number in $B_2\setminus A$. If $i_0\neq j_0$, then the $B_1\Delta B_2=\{j_1, i_0, j_2, j_0\}$ which is a contradiction. So $i_0=j_0$.  Therefore, by induction, $B_\ell\setminus A=\{i_0\}$  for each $\ell=1, 2, \ldots, n-k$. Then if $C=\{i_0, i_2, \ldots, i_k, j_1, \ldots j_{n-k}\}$, then $C\Delta A=\{i_0, i_1\}$ and $C\Delta B_\ell=\{i_1, j_\ell\}$ for each $\ell=1, 2, \ldots, n-k$, i.e., $d(C, A)=2$ and $d(C, B_\ell)=2$ for each $\ell=1, 2, \ldots, n-k$.  If $C$ is in $\sigma$, then $i_1\notin \bigcap\sigma$; and if $C$ is not in $\sigma$, then $\sigma $ is not a maximal simplex. These contradictions show that it is impossible that  $0<|\bigcap \sigma|<n-1$.

Now we suppose that $\bigcap \sigma=\emptyset$. Pick $A\in \sigma$ and represent $A$ as $i_1i_2\cdots i_n$. For each $\ell=1, 2, \ldots, n$, there exists $B_\ell \in \sigma$ such that $i_\ell \notin B_\ell$. Using the argument above, we can show that $B_\ell\setminus A=B_{\ell'}\setminus A$ for each $\ell, \ell'=1, 2, \ldots, n$. Denote $B_1\setminus A=\{i_{n+1}\}$. Then clearly $\sigma = L[i_1, i_2, \ldots, i_{n+1}]$.

To prove ii), we start with a $k$-simplex $\{A_1, A_2, \ldots, A_{k+1}\}$ in $K$ such that $|\bigcap_{\ell=1}^{k+1}A_\ell|<n-1$ and $k\geq 2$. Then if $\sigma$ is a maximal simplex in $K$ such that $\{A_1, A_2, \ldots, A_{k+1}\}\subseteq \sigma$, then $\bigcap \sigma=\emptyset$ by the argument above, and hence $\sigma$ is in the form $L[i_1, i_2, \ldots, i_{n+1}]$. Clearly $A_1\cup A_2 \subseteq \{i_1, i_2, \ldots, i_{n+1}\}$ which means  $A_1\cup A_2 = \{i_1, i_2, \ldots, i_{n+1}\}$ because $|A_1\cup A_2|=n+1$. It is clear that no other maximal simplex contains this simplex.  \end{proof}

We need one more result before the discussion of the homotopy types of the complex $\mathcal{VR}(\mathcal{F}_n^m, 2)$. Assume $n\geq 1$.  Fix a number $a\in [m]$, let $\mathcal{S}_a=\{A: A\in\mathcal{F}^m_n \text{ and } a\in A\}$. There is a natural isometric mapping between the metric spaces  $\mathcal{F}_{n-1}^{m-1}$ and  $\mathcal{S}_{a}$. Hence $\mathcal{VR}(\mathcal{F}_{n-1}^{m-1}, 2)$ is homeomorphic to $\mathcal{VR}(\mathcal{S}_{a}, 2)$. Next, we show that the homotopy type  of the star cluster of the latter in $K$ remains the same.

\begin{lemma}\label{homot_sets_contain_1} Let $n, m$ be in $\mathbb{N}$ such that $n<m$. Define  $\mathcal{S}_1=\{A\subset [m]: |A|=n \text{ and }1\in A\}$ and let $L$ be the complex $\mathcal{VR}(\mathcal{S}_1, 2)$. Then $$\text{SC}_{\mathcal{VR}(\mathcal{F}_n^m, 2)}(L)\simeq L.$$
\end{lemma}

\begin{proof} Let $K=\mathcal{VR}(\mathcal{F}_n^m, 2)$. We'll show that the condition in Lemma~\ref{SC_homo} is satisfied. Then the result follows.

Pick vertices $A$ and $B$ in $L$ such that $\{A, B\}$ is not an edge in $L$, i.e., $|A\Delta B|\geq 4$.  Hence there exist natural numbers $i_1, i_2, j_1$, and $j_2$ such that $\{i_1, i_2\}\subseteq A\setminus B$ and $\{j_1, j_2\}\subseteq B\setminus A$. Suppose, for contradiction, that $(\text{st}_K(A)\cap \text{st}_K(B))\setminus L\neq \emptyset$. We pick a vertex $C\in (\text{st}_K(A)\cap \text{st}_K(B))\setminus L$. Clearly $1\notin C$. We claim that $A\setminus \{1\}\subset C$, otherwise there exists $i_0\neq 1$ such that $i_0\in A\setminus C$, whence $|A\setminus C|\geq 2$ which is a contradiction. Similarly, $B\setminus \{1\}\subset C$. Therefore, $\{i_1, i_2, j_1, j_2\}\subset C$.  Notice that $(A\cap B)\setminus \{1\}$ has size $\leq n-3$. Suppose that $|(A\cap B)\setminus \{1\}|=n-3$. Then  the vertex  $C$ has  $n+1$ elements because $(A\cap B)\setminus \{1\}\subset C$ and $\{i_1, i_2, j_1, j_2\}\subset C$. This means that $C$ has  at least $n+1$ elements which is a contradiction. This finishes the proof.   \end{proof}

Now we are ready to give a complete characterization of  the homotopy types of $\mathcal{VR}(\mathcal{F}_n^m, 2)$.

\begin{theorem}\label{homo_equ_n} Suppose that $1<n<m-1$. The complex $\mathcal{VR}(\mathcal{F}_n^m, 2)$ is homotopy equivalent to a wedge sum of spheres. Specifically, $$\mathcal{VR}(\mathcal{F}_n^m, 2)\simeq
(\bigvee_{{m-1\choose n+1}\cdot {n \choose 2}}S^2) \vee \mathcal{VR}(\mathcal{F}_{n-1}^{m-1}, 2).$$

\end{theorem}

\begin{proof} Notice that the complex $\mathcal{VR}(\mathcal{F}_{1}^{m-1}, 2)$ is contractible. Hence the result holds when $n=2$ by Barmak's result mentioned above.

Assume that $n>2$ and $\mathcal{VR}(\mathcal{F}_{n-1}^{m-1}, 2)$ is homotopic to a wedge sum of spheres $S^2$. We denote $K=\mathcal{VR}(\mathcal{F}_n^m, 2)$. As in Lemma~\ref{homot_sets_contain_1}, let $\mathcal{S}_1=\{A\subset [m]: |A|=n \text{ and }1\in A\}$ and $L$ be the complex $\mathcal{VR}(\mathcal{S}_1, 2)$. Then, the complex $L$ is homotopy equivalent to $\mathcal{VR}(\mathcal{F}_{n-1}^{m-1}, 2)$ which is a wedge sum of $S^2$'s by the assumption. Also by Lemma~\ref{homot_sets_contain_1}, the star cluster SC$_K(L)$ is homotopy equivalent to $L$.

Now we examine the collection of maximal simplices in $K$ to decide which of them is not in SC$_K(L)$. Notice that any maximal simplex in the form $N[i_1, i_2, \ldots, i_{n-1}]$ or $L[1, i_1, \ldots, i_n]$ contains at least one vertex containing $1$ for any $i_1, i_2, \ldots, i_{n}\in [m]$; hence any such simplex is in SC$_K(L)$. Therefore  in the complement of  $\text{SC}_K(L)$, namely $K\setminus \text{SC}_K(L)$,  there is only one kind of maximal simplicies in the form $L[i_1, i_2, \ldots, i_{n+1}]$ with $i_k\neq 1$ for any $k=1, 2, \ldots, n+1$; and  there are ${m-1\choose n+1}$-many such simplices  and list them as $\{\sigma_1, \sigma_2, \ldots, \sigma_{m-1\choose n+1}\}$. Here, $K_{\sigma_\ell}$ is the complex generated by $\sigma_\ell$ for each $\ell=1, 2, \ldots, {m-1\choose n+1}$.

For each $\ell$ with $1\leq \ell\leq {m-1\choose n+1}$, we denote $L_
\ell$ to be the complex whose maximal simplices are $\{\sigma_j: j=1, 2, \ldots, \ell\}$. Hence the complex $L_{m-1\choose n+1}$ is the complex $\mathcal{VR}(\mathcal{S}_2, 2)$ where $\mathcal{S}_2$ is the collection of $n$-subsets of $[m]$ not containing $1$. Therefore, $K=\text{SC}_K(L)\cup L_{m-1\choose n+1}$.

We claim that $\text{SC}_K(L)\cup L_\ell$ is homotopic to $(\bigvee_{\ell \cdot {n \choose 2}}S^2) \vee \mathcal{VR}(\mathcal{F}_{n-1}^{m-1}, 2)$ for each $\ell=1, 2, \ldots, {m-1\choose n+1}$. This claim finishes the proof. Next, we'll prove this claim by induction. For convenience, denote $L_0=\emptyset$.

Suppose, for induction, that $$\text{SC}_K(L)\cup L_{\ell-1}\simeq (\bigvee_{(\ell-1) \cdot {n \choose 2}}S^2) \vee \mathcal{VR}(\mathcal{F}_{n-1}^{m-1}, 2).$$ This holds when $\ell=1$ since $L_0=\emptyset$. Then $\text{SC}_K(L)\cup L_{\ell}=\text{SC}_K(L)\cup L_{\ell-1}\cup \{K_{\sigma_\ell}\}$. Denote $\sigma_\ell$ to be $L[i_1, i_2, \ldots, i_{n+1}]$ where $i_k\neq 1$ for each $k=1, 2, \ldots, n+1$. Next we'll find the homotopy type of  $(\text{SC}_K(L)\cup L_{\ell-1})\cap \{K_{\sigma_\ell}\}$.

For any vertex $B\in K_{\sigma_\ell}$, $B\in L[\{1\}\cup B]\subset \text{SC}_K(L)$. Hence the $0$-skeleton of $K_{\sigma_\ell}$ is contained in $\text{SC}_K(L)$.  Let $\{B_1, B_2\}$ be a $1$-simplex in $K_{\sigma_\ell}$. Then $|B_1\cap B_2|=n-1$. Because $N[B_1\cap B_2]$ is in  $\text{SC}_K(L)$, the edge $\{B_1, B_2\}$ is in $\text{SC}_K[L]$. So the $1$-skeleton $K_{\sigma_\ell}^{(1)}$ of $K_{\sigma_\ell}$ is also contained in $\text{SC}_K(L)$. Moreover, any $k$-simplex with $k\geq 2$ in $K_{\sigma_\ell}$ is not in  $\text{SC}_{K}(L)$; otherwise such a $k$-simplex would be contained in a maximal simplex which has a vertex containing $1$ and hence is different from $\sigma_\ell$.  This leads to a contradiction by ii) in Lemma~\ref{max_simp}. For any $\ell'=1, 2, \ldots, \ell-1$, the intersection of the complexes $K_{\sigma_{\ell'}}$ and $K_{\sigma_\ell}$ contains at most one vertex because of their definitions. Therefore, $(\text{SC}_K(L)\cup L_{\ell-1})\cap K_{\sigma_\ell}=K_{\sigma_\ell}^{(1)}$. Recall that $\sigma_\ell$ is an $n$-simplex, hence $K_{\sigma_\ell}^{(1)}$ is homotopy equivalent to a wedge sum of  ${n\choose 2}$-many copies of $S^1$'s by Lemma~\ref{subcom_simplex}.

Notice that $K_{\sigma_\ell}^{(1)}$ is null-homotopic in $K_{\sigma_\ell}$ because $K_{\sigma_\ell}$ is contractible. Also,  $K_{\sigma_\ell}^{(1)}$ is null-homotopic in $\text{SC}_K(L)\cup L_{\ell-1}$ because the homotopy type of  former is a wedge sum of $S^1$'s and the homotopy type of latter is a wedge sum of $S^2$'s. Therefore by Lemma~\ref{cup_simp}, $\text{SC}_K(L)\cup L_\ell$ is homotopy equivalent to $\Sigma(\bigvee_{n\choose 2}S^1)\vee (\text{SC}_K(L)\cup L_{\ell-1})$ which is by inductive assumption $(\vee_{\ell {n\choose 2}}S^2)\vee \text{SC}_K(L)$.  This finishes the proof because $\text{SC}_K(L)\simeq L\simeq \mathcal{VR}(\mathcal{F}_{n-1}^{m-1}, 2)$.
\end{proof}

By an inductive calculation, we obtain the following corollary.
\begin{cor}\label{homot_n_2} Suppose that $1<n<m-1$. The  complex  $\mathcal{VR}(\mathcal{F}_n^m, 2)$ is homotopy equivalent to a wedge sum of $\sum_{k=2}^n{m+k-1-n\choose k+1}{k\choose 2}$-many copies of $S^2$'s.
\end{cor}

\section{Vietoris-Rips Complex $\mathcal{VR}(\mathcal{F}_{\preceq A}^{m},2)$}\label{F_A_m}

In this section, we'll determine the homotopy type of $\mathcal{VR}(\mathcal{F}_{\preceq A}, 2)$ for $A\in \mathcal{P}([m])$ with $|A|=n$.

As in the discussion in Section~\ref{Intro}, $\mathcal{VR}(\mathcal{F}^m_{\leq r}, r)$ is a cone with cone vertex being the empty set, hence contractible;
%with vertex $\emptyset$, i.e. $\emptyset\ast \mathcal{F}_{1\leq n\leq r}$,
and similarly $\mathcal{VR}(\mathcal{F}^m_{\geq m-r}, r)$ is also contractible.  Hence, for any $A\subset [m]$ with $|A|\leq 2$, the complex $\mathcal{VR}(\mathcal{F}^m_{\preceq A}, 2)$ is contractible. So in this section, we will discuss the homotopy type of $\mathcal{VR}(\mathcal{F}^m_{\preceq A}, 2)$ with $|A|\geq 3$.

The following lemma is easy to prove, but heavily used in the discussion of $\mathcal{VR}(\mathcal{F}_{\preceq A}^m, 2)$.

\begin{lemma}\label{subset_dist2} For any $A, B\in \mathcal{P}[m]$ with $|A|<|B|$, $d(A, B)\leq 2$ if and only if $A\subset B$ and $|B\setminus A|\leq 2$.

\end{lemma}
\begin{proof} If $A\subset B$ and $|B\setminus A|\leq 2$, then $d(A, B)=|(A\setminus B)\cup (B\setminus A)|\leq 2$.

Now we suppose $A\setminus B\neq \emptyset$, i.e. $|A\setminus B|\geq 1$. Since $|A|<|B|$, $|B\setminus A|\geq 2$, therefore $d(A, B)\geq 3$. If $A\subset B$ and $|B\setminus A|>2$, then $d(A, B)=|B\setminus A|>2$. This finishes the proof.
\end{proof}

Next, we'll discuss the homotopy type of $\mathcal{VR}(\mathcal{F}_n^m\cup \mathcal{F}_{n+1}^m, 2)$ using a similar approach as in the proof of Theorem~\ref{homo_equ_n}.

\begin{theorem}\label{homot_wide_2_2} Suppose that $1<n<m-1$. Then the complex $\mathcal{VR}(\mathcal{F}_n^m\cup \mathcal{F}_{n+1}^m, 2)$ is homotopy equivalent to a wedge sum of $(\sum_{k=2}^n {m+k-1-n\choose k+1}\cdot{k\choose 2}+{m\choose n+2}\cdot {n+1\choose 2})$-many copies of $S^2$.
\end{theorem}

\begin{proof} Let $K=\mathcal{VR}(\mathcal{F}_n^m\cup \mathcal{F}_{n+1}^m, 2)$ and $K_0=\mathcal{VR}(\mathcal{F}_n^m, 2)$. By Corollary~\ref{homot_n_2}, the complex $K_0$ is homotopy equivalent to a wedge sum of  $\sum_{k=2}^n {m+k-1-n\choose k+1}\cdot{k\choose 2}$-many copies of $S^2$'s.

We claim that $\text{SC}_K(K_0)\simeq K_0$. We proceed to show that the condition in Lemma~\ref{SC_homo} is satisfied. Hence this claim holds. Take a $B\in \mathcal{F}_{n+1}^m$ such that $B\in \text{st}_K(D)\cap \text{st}_K(D')$ for $D, D'\in \mathcal{F}_n^m$. Then $d(B, D)=d(B, D')=2$, hence by Lemma~\ref{subset_dist2}, $D, D'$ are both subsets of $B$ which implies that $d(D, D')=2$. This finishes the proof of the claim.

By Lemma~\ref{max_simp}, there are two types of maximal simplicies in $\mathcal{VR}(\mathcal{F}_{n+1}^m, 2)$.  If $\sigma$ is a maximal simplex $\mathcal{VR}(\mathcal{F}_{n+1}^m, 2)$ which can be represented in the form $N[i_1, i_2, \ldots, i_n]$, clearly $\{i_1i_2\cdots i_n\}\cup N[i_1, i_2, \ldots, i_n]$ is a simplex in $K$; hence $N[i_1, i_2, \ldots, i_n]\in\text{SC}_K(K_0)$.

Now we look at the second type of maximal simplices in $\mathcal{VR}(\mathcal{F}_{n+1}^m, 2)$. There are ${m\choose n+2}$-many type of maximal simplicies in $\mathcal{VR}(\mathcal{F}_{n+1}^m, 2)$ which are in the form $L[i_1, i_2, \ldots, i_{n+2}]$; and list such $(n+1)$-simplicies as $\{\sigma_1, \sigma_2, \ldots, \sigma_{{m\choose n+2}}\}$. Denote $L_\ell =\text{SC}_K(K_0)\cup \bigcup_{j=1}^\ell K_{\sigma_j}$ for $\ell=1, 2, \ldots, {m\choose n+2}$. Recall that the complex $K_{\sigma_j}$ is the complex generated by the simplex $\sigma_j$ for $j=1, 2, \ldots, {m\choose n+2}$.

Assume for induction that $L_{\ell-1}$ is homotopic to $$\bigvee_{\sum_{k=2}^n {m+k-1-n\choose k+1}\cdot{k\choose 2}+(\ell-1)\cdot {n+1\choose 2}}S^2.$$
This is clearly true when $\ell=1$. We claim that $L_{\ell-1}\cap K_{\sigma_\ell} =K_{\sigma_\ell}^{(1)}$ which is homotopic to $\bigvee_{n+1\choose 2}S^1$ and hence is  null-homotopic in both $L_{\ell-1}$ and $K_{\sigma_\ell}$. By Lemma~\ref{cup_simp}, this implies that $L_\ell$ is homotopy equivalent to a wedge sum of $(\sum_{k=2}^n {m+k-1-n\choose k+1}\cdot{k\choose 2}+\ell\cdot {n+1\choose 2})$-many $S^2$. This finishes the proof.  Next, we'll prove our claim.

By part ii) of Lemma~\ref{max_simp}, any $2$-simplex in $K_{\sigma_\ell}$ is not in $L_{\ell-1}$. Let $\{B_1, B_2\}$ be a $1$-simplex in $K_{\sigma_\ell}$. Then $B_1\cap B_2$ is an $n$-subset, i.e., a vertex in $K_0$; so $\{B_1, B_2, B_1\cap B_2\}$ is a $2$-simplex in $K$ which means $\{B_1, B_2\}\in \text{st}_K(B_1\cap B_2)$. This shows that $L_{\ell-1}\cap K_{\sigma_\ell} =K_{\sigma_\ell}^{(1)}$. \end{proof}

To identify the homotopy types of $K=\mathcal{VR}(\mathcal{F}^m_{\preceq A}, 2)$ with $|A|\geq 3$, we'll use Lemma~\ref{complex_add_1v} by taking the vertex $A$ so that $K= (K\setminus A)\cup \text{st}_K(A)$. So the key is to understand the link of $A$ in $K$, lk$_K(A)$. Next lemma shows that lk$_K(A)$ is a wedge sum of $S^2$'s.

Note that when $n=3$, $\sum_{k=2}^{n-2}{k\choose 2}$ is set to be $0$ as introduced in Section~\ref{prel}.
\begin{lemma}\label{lk_precA_whole} Suppose that $m\geq n>2$ and $A=i_1i_2\cdots i_n\in \mathcal{P}([m])$.

%\begin{itemize}

%\item[i)] Let $\sigma$ be a maximal simplex in $\text{lk}_{\mathcal{VR}(\mathcal{F}^m_{\preceq A^\ast},2)}(A^\ast)$ is its intersection with one of the following: a), $L[B]\cup \{B\}\cup N[B]$ where $B$ is an $(n-1)$-subset of $A^\ast$; b) $N(B)$ where $B$ is an $(n-3)$-subset of $A^\ast$; c), $L[A^\ast]$; or d), $L[B\cup \{i\}]$ where $B$ is an  $(n-1)$-subset and $i\in [m]\setminus A^\ast$.

%\item[ii)]

Denote $i_0=-1$ and define $d_\ell=i_\ell-(i_{\ell-1}+1)$ for each $\ell=1, 2, \dots, n$.  Then

$$\text{lk}_{\mathcal{VR}(\mathcal{F}^m_{\preceq A},2)}(A)\simeq \bigvee_{\sum_{k=2}^{n-2}{k\choose 2}+\sum_{\ell=1}^{n-2}d_{\ell}\cdot {n-\ell \choose 2}} S^2.$$

%\end{itemize}
\end{lemma}

\begin{proof} Let $K=\mathcal{VR}(\mathcal{F}^m_{\preceq A},2)$. Note that for any $B$ with $|B|\leq n-3$, $d(A, B)\geq 3$. Next we divide the vertices in the link of the vertex $A$ in $K$, lk$_K(A)$,  into the following pairwise disjoint collections $\mathcal{G}_k$ for $k=0, 1, \ldots, i_{n-1}$. %and $\mathcal{H}_\ell$ for $\ell =2, 3, \ldots, n$.
These collections are defined as the following:

\begin{itemize}

\item[i)] $\mathcal{G}_{0}=\{B\in \mathcal{P}:|B|<n \text{ and }  d(B, A)=2\}$;

\item[ii)] for $k\in \{1, 2, \ldots, i_{n-1}\}\setminus \{i_1, i_2, \dots,i_{n-1}\}$ , $\mathcal{G}_k$ contains all the $B's$ with $|B|=n$ such that  $B$ contains $k$, all $i_j$'s with $i_j<k$, all but one of $i_j$'s with $i_j>k$;  %for $\ell =2, 3, \ldots, n$, $\mathcal{H}_\ell$ is the collection of the $n$-subsets of $[m]$ obtained from $A$ through replacing one of the numbers $\{i_\ell, \ldots, i_n\}$ by a number between $i_{\ell-1}$ and $i_\ell$.  Hence if $d_\ell=0$, $\mathcal{H}_\ell=\emptyset$.
\item[iii)] $\mathcal{G}_{i_{n-1}}$ contains all the $B$'s with $|B|=n$ such that $\{i_1, i_2, \ldots, i_{n-1}\}\subset B$ and $B$ contains any other number between $i_{n-1}$ and $i_n$.

\item[iv)] $\mathcal{G}_{i_{j}}=\emptyset$ for $j=1, 2, \dots, n-2$ for the purpose of convenience.
\end{itemize}
By Lemma~\ref{subset_dist2}, $\mathcal{G}_0$ contains all the $B$'s such that $B\subset A$ and $|B| = n-1$ or $n-2$. Also, %$\bigcup_{k=1}^{i_1-1}\mathcal{G}_k\cup \bigcup_{\ell=2}^{n}\mathcal{H}_\ell$
it is clear that $\bigcup_{k=1}^{i_{n-1}}\mathcal{G}_k$ contains all the $B$'s such that $B\prec A$, $d(A, B)=2$, and $|B|=n$. Hence $\text{lk}_{K}(A) =\mathcal{VR}(\bigcup_{k=0}^{i_{n-1}}{\mathcal{G}_k}, 2)$. For each $k=0, 1, \ldots, i_{n-1}$, we define $K_k=\mathcal{VR}(\mathcal{G}_k, 2)$  if $\mathcal{G}_k\neq \emptyset$ and $K_{\leq k}=\mathcal{VR}(\bigcup_{i=0}^k\mathcal{G}_i, 2)$. %For $\ell=2,3, \ldots, n$, we define $L_\ell=\mathcal{VR}(\mathcal{H}_\ell, 2)$ and $L_{\leq \ell}=\mathcal{VR}(\bigcup_{i=0}^{i_1-1}\mathcal{G}_i\cup \bigcup_{i=2}^{\ell} \mathcal{H}_\ell, 2)$.
Hence $\text{lk}_K(A)=K_{\leq i_{n-1}}$.

Since $\mathcal{G}_{0}$ is the collection of all $(n-2)$-subsets and $(n-1)$-subsets of the $n$-set $A$, the complex $K_{0}$ is homeomorphic to $\mathcal{VR}(\mathcal{F}_{n-2}^n\cup \mathcal{F}^n_{n-1}, 2)$; hence by Theorem~\ref{homot_wide_2_2}, the complex $K_{0}=K_{\leq 0}$ is homotopy equivalent to a wedge sum of $(\sum_{k=2}^{n-2}{k\choose 2}+{n-1\choose 2})$-many copies of $S^2$'s.
%For the sake of convenience, we denote $i_0=0$.
Since  $\mathcal{G}_{i_j}=\emptyset$ for $j=1, 2, \ldots, n-2$, the complex $K_{\leq i_j}$ is same as $K_{\leq i_j-1}$ for such $j$.

Now we investigate the complex $K_k$ with $k\geq 1$ and the collection $\mathcal{G}_k\neq \emptyset$. Fix $k$ such that $1\leq k<i_{n-1}$ and $\mathcal{G}_k\neq \emptyset$. Then there exists an $\ell$ in the set  $\{1, 2, \ldots, n-1\}$ such that $i_{\ell-1}<k<i_\ell$. Then, the complex $K_k$ is the complex generated by a proper face of $L[i_1, \ldots, i_{\ell-1}, k, i_\ell, \ldots, i_n]$ which consists of all $B$ which contains $\{i_1, \ldots, i_{\ell-1}, k\}$ and all but one of $\{i_\ell, \ldots, i_n\}$; hence it is an $(n-\ell)$-simplex. And $K_{i_{n-1}}$ is a proper face of $N[i_1, i_2, \ldots, i_{n-1}]$ which includes all $B$'s which contains $\{i_1, i_2, \ldots, i_{n-1}\}$ and another number between $i_{n-1}$ and $i_n$; hence it is a complex generated by a $(d_n-1)$-simplex.

Next we determine the homotopy type of $K_{\leq i_{n-2}}$. If there is no $k$ such that $k\in [i_{n-2}]\setminus \{i_1, i_2, \ldots, i_{n-2}\}$, then  $d_1=1$ and  $d_2, \ldots, d_{n-2}$ are all zeroes and the complex $K_{\leq i_{n-2}}=K_0$ which is clearly homotopy equivalent $\bigvee_{\sum_{k=2}^{n-2}{k\choose 2}+\sum_{\ell=1}^{n-2}d_{\ell}\cdot {n-\ell \choose 2}} S^2$.  Now we suppose otherwise and fix $k$ such that $1\leq k\leq  i_{n-2}$ and $i_{\ell-1}<k<i_{\ell}$ for some $\ell=1, 2, \ldots, n-2$, here we define $i_0=0$.   Suppose, for induction, that $K_{\leq (k-1)}$ is homotopy equivalent to a wedge sum of %$(\sum_{j=2}^{n-2}{j\choose 2}+k{n-1\choose 2})$-many copies of
 $S^2$'s. This holds when $k$ is the minimal natural number different from $i_1, i_2, \ldots i_{n-2}$ in which case $K_{\leq k-1}$ is homotopy equivalent to $K_0$. By Lemma~\ref{SC_complex}, $K_{\leq k}=\text{SC}_{K_{\leq k}}(K_{\leq (k-1)}) \cup K_{k}$. We'll prove the following two claims and these two claims  imply that $K_{\leq k}\simeq K_{\leq {(k-1)}}\vee (\bigvee_{{n-\ell \choose 2}} S^2)$  by Lemma~\ref{cup_simp} and the inductive assumption.

\begin{description}
\item[Claim i)] $\text{SC}_{K_{\leq k}}(K_{\leq (k-1)})\simeq K_{\leq (k-1)}$.

\item[Claim ii)] $\text{SC}_{K_{\leq k}}(K_{\leq (k-1)}) \cap K_{k}\simeq \bigvee_{{n-\ell\choose 2}}S^1$.
\end{description}

\textbf{Proof of Claim i):} We'll verify that the condition in Lemma~\ref{SC_homo} is satisfied. Then the result follows. We'll show that $d(C_1, C_2)\leq 2$ for $C_1, C_2 \in K_{\leq (k-1)}$ whenever $st_{K_{\leq k}}(C_1)\cap st_{K_{\leq k}}(C_2)\setminus K_{\leq (k-1)}\neq \emptyset$.  Pick a vertex $D$ in $K_k$. Then $D$ contains $k$ and an $(n-1)$-subset of $A$, denoted by $C$. Then for any vertex $B\in K_{\leq (k-1)}$, $D\in \text{st}_{K_{\leq k}}(B)$ if and only if $B$ is one of the following: a) $C\subset B$ and $B$ contains one of $1, 2, \ldots, k-1$ not in $A$; b) $C$; c) any $(n-2)$ subset of $C$. Any pair of such vertices  have distance $2$; hence they form a $1$-simplex in $K_{\leq (k-1)}$. This finishes the proof of Claim i).

\textbf{Proof of Claim ii):} Since the complex $K_k$ is generated by an $(n-\ell)$-simplex, $K_k^{(1)}$ is homotopy equivalent to $\bigvee_{{n-\ell\choose 2}}S^1$ by Lemma~\ref{subcom_simplex}. We'll show that  $\text{SC}_{K_{\leq k}}(K_{\leq (k-1)}) \cap K_{k}=K_k^{(1)}$. Pick any pair of vertices, $B_1, B_2$, in $K_k$. Then,  $B_1\cap B_2$ contains the number $k$ and an $(n-2)$-subset of $A$, denoted by $D$. Note that $D$ is a vertex in the complex $K_{0}\subseteq K_{\leq k-1}$; therefore, the $1$-simplex $\{B_1, B_2\}\in \text{st}_{K_{\leq k}}(D)$. Hence $K_k^{(1)}\subseteq \text{SC}_{K_{\leq k}}(K_{\leq (k-1)}) \cap K_{k}$. It is straightforward to verity that for any $B\in\mathcal{G}_i$ with $i=1, 2, \ldots, k-1$,  $\text{st}_{K_{\leq k}}(B)\cap K_{k}$ is a complex containing only one vertex because any vertex in this complex must contains $B\cap A$ and the number $k$. Similarly, for any $B\in\mathcal{G}_{0}$ with $|B|=n-1$, there is at most $1$ vertex in $K_k$ containing $B$ as a subset, i.e. having a distance $\leq 2$ from $B$;  and if $B\in\mathcal{G}_{0}$ with $|B|=n-2$, then there are at most two vertices in $K_k$ which have distance $2$ from $B$. Hence, $\text{st}_{K\leq k}(B)\cap K_k\subseteq K_k^{(1)}$ for any vertex $B$ in the complex $K_{\leq (k-1)}$.    This finishes the proof of Claim ii).

\medskip

By an inductive calculation, we have proved that the complex $K_{\leq i_{n-2}}$  is homotopy equivalent to a wedge sum of  $(\sum_{k=2}^{n-2}{k\choose 2}+\sum_{\ell=1}^{n-2}d_{\ell}\cdot {n-\ell \choose 2})$-many $S^2$'s. Next, we show that the complex $K_{\leq (i_{n-1}-1)}$ is homotopy equivalent to $K_{\leq i_{n-2}}$.
If $d_{n-1}=0$, then $K_{<i_{n-1}}= K_{\leq i_{n-2}}$; otherwise we fix $k$ with $i_{n-2}<k<i_{n-1}$ and suppose that $K_{\leq (k-1)}\simeq K_{\leq i_{n-2}}$. The collection  $\mathcal{G}_k$ contains two vertices $i_1i_2\cdots i_{n-2}ki_n$ and $i_1i_2\cdots i_{n-2}ki_{n-1}$; and  the simplex $\{i_1i_2\cdots i_{n-2}ki_n, i_1i_2\cdots i_{n-2}ki_{n-1}\}$ is in $\text{st}_{ K_{\leq (i_{n-1}-1)}}(D)$ where $D=i_1i_2\cdots i_{n-2}\in K_{\leq (k-1)}$. Hence SC$_{K_{\leq k}}(K_{\leq (k-1)})=K_{\leq k}$. By a similar discussion as in the proof of claim i), we can verify that the complex $K_{\leq k}$ and its subcomplex $K_{\leq (k-1)}$ satisfy the condition in Lemma~\ref{SC_homo}. Therefore, SC$_{K_{\leq k}}(K_{\leq (k-1)})\simeq K_{\leq (k-1)}$. Hence the complex $K_{\leq k}$ is homotopy equivalent to $K_{\leq i_{n-2}}$. Therefore by induction, the complex $K_{\leq (i_{n-1}-1)}$ is homotopy equivalent to $K_{\leq i_{n-2}}$.

In the last part, we show that the complex $K_{\leq i_{n-1}}=\text{lk}_K(A)$ is also homotopy equivalent to $K_{\leq i_{n-2}}$. Again it is straightforward to verify that $K_{\leq i_{n-1}}$ and $K_{\leq (i_{n-1}-1)}$ satisfy the condition in Lemma~\ref{SC_homo}. Hence, SC$_{K_{\leq i_{n-1}}}(K_{\leq (i_{n-1}-1)})\simeq K_{\leq (i_{n-1}-1)}$. Recall that  $K_{i_{n-1}}$ is a complex generated by a proper face of the simplex $N[i_1i_2\cdots i_{n-1}]$. Note that $i_1i_2\cdots i_{n-1}$ is a vertex in $K_{\leq (i_{n-1}-1)}$; and also, $\{i_1i_2\cdots i_{n-1}\}\cup \mathcal{G}_{i_{n-1}}$ is a simplex in $K_{\leq i_{n-1}}$. So, $K_{i_{n-1}}\subset \text{st}_{K_{\leq i_{n-1}}}(i_1i_2\cdots i_{n-1})$; and hence SC$_{K_{\leq i_{n-1}}}(K_{\leq (i_{n-1}-1)})=K_{\leq i_{n-1}}$. And this finishes the proof.  \end{proof}
%So far, we proved that $K_{\leq i_1-1}$ is homotopy equivalent to $\bigvee_{(\sum_{j=2}^{n-2}{j\choose 2}+d_1{n-1\choose 2})}S^2$ due to $d_1=i_1$. For $\ell=2, 3, \ldots, n$ with $d_\ell \neq 0$, we divide $\mathcal{H}_\ell$ into the collections

Motivated by the lemma above, we define a natural number $r_A$ for each $A\subset [m]$ in the following way.  For each $A=i_1 i_2 \cdots i_n\subseteq [m]$ with $d_1=i_1$ and $d_\ell=i_\ell-(i_{\ell-1}+1)$ for $\ell=2, 3, \ldots, n$, we define $$r_A= \sum_{k=2}^{n-2}{k\choose 2}+\sum_{\ell=1}^{n-2}d_{\ell}\cdot {n-\ell \choose 2}.$$

\begin{theorem}\label{homo_leqn} Suppose that $m\geq n>2$ and $A=i_1i_2\cdots i_n\in \mathcal{P}([m])$. Then the complex $\mathcal{VR}(\mathcal{F}^m_{\preceq A},2)$  is homotopy equivalent to a wedge sum of $S^3$'s.

More specifically,  if $A$ is the vertex $\{1, 2, 3\}\subset[m]$, $$\mathcal{VR}(\mathcal{F}^m_{\preceq A},2)\simeq S^3.$$

And for any other vertex $A$ with $\{1, 2, 3\}\prec A$,
$$\mathcal{VR}(\mathcal{F}^m_{\preceq A},2)\simeq (\bigvee_{r_A} S^3)\vee \mathcal{VR}(\mathcal{F}^m_{\prec A},2). $$
Therefore if $A\in \mathcal{P}([m])$ with $|A|\geq 3$, $\mathcal{VR}(\mathcal{F}^m_{\preceq A},2)$ is homotopy equivalent to the wedge sum of $\sum\{r_B: \{1, 2, 3\}\preceq B\preceq A\}$-many copies of $S^3$.
\end{theorem}

\begin{proof} Let $K=\mathcal{VR}(\mathcal{F}^m_{\preceq A},2)$ and $L=\mathcal{VR}(\mathcal{F}^m_{\prec A},2)$. Suppose $A=\{1, 2, 3\}$. Then $r_A=1$, hence $\text{lk}_K(A)$ is homotopic to $S^2$ by Lemma~\ref{lk_precA_whole}. Because the complex $L$ is contractible,  the complex $K$ is homotopy equivalent to $S^3$ by  Lemma~\ref{complex_add_1v}.

Fix $A$ with $\{1, 2, 3\}\prec A$ and suppose for induction that $L$ is homotopy equivalent to a wedge sum of $S^3$'s. Again by Lemma~\ref{lk_precA_whole}, $\text{lk}_K(A)$ is homotopic to a wedge sum of $r_A$-many $S^2$'s. Hence the inclusion map from $\text{lk}_K(A)$ to $L$ is null-homotopic. Therefore, the general result holds due to again Lemma~\ref{complex_add_1v}. \end{proof}

The following result is a direct application of Lemma~\ref{cup_simp}, Lemma~\ref{lk_precA_whole}, and Theorem~\ref{homo_leqn}.
\begin{theorem}\label{allpowersets}Suppose that $m\geq n>2$. For each $n$, we define $$t_n=\sum_{A\subseteq[m]\text{ with } |A|=n}r_A.$$
Then,
%{m-1\choose n-1}\cdot {n\choose 3}+2{m-2\choose n-1}\cdot{n \choose 3}+\cdots +(m-n+1){n-1\choose n-1}\cdot {n\choose 3}=
$$\mathcal{VR}(\mathcal{F}^m_{\leq n},2)\simeq \bigvee_{t_n} S^3\vee \mathcal{VR}(\mathcal{F}^m_{\leq n-1},2). $$

Therefore, $\mathcal{VR}(\mathcal{F}^m_{\leq n},2)$ is homotopy equivalent to the wedge sum of ($\sum_{k=3}^n t_k$)-many copies of $S^3$.

\end{theorem}

Adamaszek and Adams in \cite{AA22} proved that $\mathcal{VR}(Q_m,2)= \mathcal{VR}(\mathcal{F}^m_{\leq m},2)\simeq \bigvee_{c_m} S^3$ for any $m>2$, where $c_m=\sum_{0\leq j<i<m}(j+1)(2^{m-2}-2^{i-1})$. By Theorem~\ref{allpowersets}, $c_m=\sum_{k=3}^mt_k$ where $t_n$ is defined as in the statement of Theorem~\ref{allpowersets}.

\section{Vietoris-Rips Complex $\mathcal{VR}(\mathcal{F}_{n}^m\cup \mathcal{F}_{n'}^m, 2)$}\label{F_p_q_m}

In this section, we'll investigate the homotopy types of $\mathcal{VR}(\mathcal{F}_{n}^m\cup \mathcal{F}_{n'}^m, 2)$ with $n, n'\in \mathbb{N}$. Clearly when $|n-n'|\geq 3$, then $\mathcal{VR}(\mathcal{F}_{n}^m\cup \mathcal{F}_{n'}^m, 2)$ is a disjoint union of
$\mathcal{VR}(\mathcal{F}_{n}^m, 2)$ and $\mathcal{VR}(\mathcal{F}_{n'}^m, 2)$; then by the discussion in Section~\ref{F_n_m}, its homotopy type is clear.  The homotopy types of the complex $\mathcal{VR}(\mathcal{F}_{n}^m\cup \mathcal{F}_{n+1}^m, 2)$ are discussed in Section~\ref{F_A_m} (see Theorem~\ref{homot_wide_2_2}).

In the following, we'll find the homotopy types of the Vietoris-Rips complexes  $\mathcal{VR}(\mathcal{F}_n^{m}\cup \mathcal{F}_{n+2}^m, 2)$ for $n+2\leq m$. Clearly for $m\geq 3$, $\mathcal{VR}(\mathcal{F}_0^{m}\cup \mathcal{F}_{2}^m, 2)$ and $\mathcal{VR}(\mathcal{F}_{m}^{m}\cup \mathcal{F}_{m-2}^m, 2)$ are contractible because both of them are cones. Next, we'll discuss the complexes $\mathcal{VR}(\mathcal{F}_n^{m}\cup \mathcal{F}_{n+2}^m, 2)$ in general.

The next result can be obtained by applying the proof of Lemma~\ref{lk_precA_whole} with small modifications; next we'll go through the difference of the proofs. For each $A=i_1i_2\cdots i_{n}\in \mathcal{F}_{n}^m$ with $c_1=i_1-1$ and $c_\ell =i_{\ell}-(i_{\ell-1}+1)$ for $\ell=2, 3, \ldots, n$, we define $$s_A= \sum_{k=2}^{n-2}{k\choose 2}+\sum_{\ell=1}^{n-2} c_\ell {n-\ell\choose 2}.$$

Note that for any $A\subset [m]$ with $|A|=n$, $r_A=s_A+{n-1\choose 2}$.

\begin{lemma}\label{lk_wide_2} Suppose that $4 \leq  n< m-1$ and $A=i_1i_2\cdots i_{n}\subset [m]$ with $i_1\geq 2$. Let $K = \mathcal{VR}(\mathcal{F}^m_{n-2} \cup \mathcal{F}^m_{n}, 2)\cap \mathcal{VR}(\mathcal{F}^m_{\preceq A}, 2)$.

Then,

$$\text{lk}_K(A)\simeq \bigvee_{s_A}S^2$$

\end{lemma}

\begin{proof}  As in the proof of Lemma~\ref{lk_precA_whole}, we divide the vertices in lk$_K(A)$ into pairwise disjoint collections $\mathcal{G}_k$ for $k =0, 1, \ldots, i_{n-1}$. For $k\geq 1$, $\mathcal{G}_k$ is exactly defined in the same way as in the proof of Lemma~\ref{lk_precA_whole}. Note that the vertices in $K$ have either size $n-2$ or $n$. Then $\mathcal{G}_0$ contains all the subsets of $A$ with size $n-2$. The complexes $K_k$ and $K_{\leq k}$ are defined in the same ways as in the proof of Lemma~\ref{lk_precA_whole} for $k=0, 1, \ldots, i_{n-1}$. Hence $K_0$ is homeomorphic to $\mathcal{VR}(\mathcal{F}_{n-2}^{n}, 2)$; hence by Corollary~\ref{homot_n_2}, it is homotopy equivalent to $\sum_{k=2}^{n-2}{k\choose 2}$-many copies of $S^2$'s. Then the rest of the proof is same as the proof of Lemma~\ref{lk_precA_whole}.      \end{proof}

\begin{theorem} \label{homo_2} Suppose that $1\leq n<m-3$. Then the complex $\mathcal{VR}(\mathcal{F}_n^{m}\cup \mathcal{F}_{n+2}^m, 2)$ is homotopy equivalent to a wedge sum of $S^3$'s.

More specifically,

$$\mathcal{VR}(\mathcal{F}_1^{m}\cup \mathcal{F}_{3}^m, 2)\simeq \bigvee_{m\choose 4} S^3;$$

and for $n\geq 2$,  % $$\mathcal{VR}(\mathcal{F}_n^{m}\cup \mathcal{F}_{n+2}^m, 2)\simeq \bigvee_{k=n+1}^{m-1} \mathcal{VR}(\mathcal{F}^k_{n-1}\cup \mathcal{F}^k_{n+1}, 2);$$
%Another formula,
we define $o_{m, n} = \sum \{s_A: A\in\mathcal{F}^m_{n+2} \text{ with }\min A\geq 2\} $ and then
$$\mathcal{VR}(\mathcal{F}_n^{m}\cup \mathcal{F}_{n+2}^m, 2)\simeq\mathcal{VR}(\mathcal{F}^{m-1}_{n-1}\cup \mathcal{F}^{m-1}_{n+1}, 2)\vee \bigvee_{o_{m, n}}S^3.$$

Therefore, $\mathcal{VR}(\mathcal{F}_n^{m}\cup \mathcal{F}_{n+2}^m, 2)$ is homotopy equivalent to $(\sum_{k=2}^n o_{m+k-n, k} +{m+1-n\choose 4})$-many copies of $S^3$.
\end{theorem}

\begin{proof} We firstly prove that $K=\mathcal{VR}(\mathcal{F}_1^{m}\cup \mathcal{F}_{3}^m, 2)\simeq \bigvee_{m\choose 4} S^3$. Let $L_0= \mathcal{VR}(\mathcal{F}_1^{m}, 2)$ which is a complex generated by a simplex because each pair of singlton subsets of $[m]$ has distance $2$. Hence by Lemma~\ref{SC_homo}, $\text{SC}_K(L_0)$ is contractible. By Lemma~\ref{max_simp}, there are two types of maximal simplices in $\mathcal{VR}( \mathcal{F}_{3}^m, 2)$, namely $N[i_1, i_2]$ and $L[i_1, i_2, i_3, i_4]$ for some $i_1, \ldots, i_4\in [m]$; clearly $\{i_1\}\cup N[i_1, i_2]\}$ is a simplex in $K$. Hence $N[i_1, i_2]\in \text{SC}_K(L_0)$ for each $i_1, i_2\in[m]$. Within $\mathcal{VR}( \mathcal{F}_{3}^m, 2)$, there are ${m\choose 4}$-many simplices in the form $L[i_1, i_2, i_3, i_4]$ and  the intersection of each pair of such simplices contains at most one vertex. We list such simplices as $\{\sigma_\ell: \ell =1, 2, \ldots, {m\choose 4}\}$ and define $L_\ell=\text{SC}_K(L_0)\cup \bigcup_{i=1}^\ell\sigma_\ell$. We see that $\sigma_\ell\notin \text{SC}_K(L_0)$ for each $\ell =1, 2, \ldots, {m\choose 4}$  because otherwise there is a number in $\cap\sigma_\ell$ which is a contradiction; and because each of $\sigma_\ell$'s proper faces has an nonempty intersection, we get that $\sigma_\ell^{(2)}\subset \text{SC}_K(L_0)$. Hence $L_{\ell-1}\cap\sigma_
\ell=\sigma_\ell^{(2)}\simeq S^2$. Therefore, by Lemma~\ref{cup_simp}, $L_1\simeq S^3$ and inductively $L_\ell\simeq \bigvee_{\ell}S^3$. This finishes the proof of first part.

Now we assume that $n\geq 2$ and $\mathcal{VR}(\mathcal{F}^{m-1}_{n-1}\cup \mathcal{F}^{m-1}_{n+1}, 2)$  is homotopy equivalent to a wedge sum of $S^3$'s. Let $\mathcal{G}_0=\{B\in\mathcal{F}_{n}^m\cup \mathcal{F}_{n+2}^m: 1\in B\}$ and $K_0=\mathcal{VR}(\mathcal{G}_0, 2)$; by a straightforward isometric mapping, we see that $K_0\cong\mathcal{VR}(\mathcal{F}_{n-1}^{m-1}\cup \mathcal{F}_{n+1}^{m-1}, 2)$ which is  homotopy equivalent to a wedge sum of $S^3$'s by the assumption.  Let $\mathcal{G}_1=\{B\in\mathcal{F}_{n}^m\cup \mathcal{F}_{n+2}^m:|B|=n \text{ or } 1\in B\}$ and $K_1=\mathcal{VR}(\mathcal{G}_1, 2)$.

Next, we show that $K_1=\text{SC}_{K_1}(K_0)\simeq K_0$. Let $\sigma$ be a simplex in $K_1$ consisting of vertices not containing $1$.  Then $\sigma$ is  a face of either $N[i_1, i_2, \ldots, i_{n-1}]$ or $L[i_1, i_2, \ldots, i_{n+1}]$ with all the numbers $>1$. Since $1i_1i_1\cdots i_{n-1}\in N[i_1, i_2, \ldots, i_{n-1}]$, $N[i_1, i_2, \ldots, i_{n-1}]\in \text{SC}_{K_1}(K_0)$. Also notice that $\{1 i_1 i_2\cdots i_{n+1}\}\cup L[i_1, i_2, \ldots, i_{n+1}]$ is a simplex in $K_1$; hence $ L[i_1, i_2, \ldots, i_{n+1}]\in \text{SC}_{K_1}(K_0)$. Therefore, $K_1=\text{SC}_{K_1}(K_0)$.

Next we show that the condition in Lemma~\ref{SC_homo} is satisfied which implies that $\text{SC}_{K_1}(K_0)\simeq K_0$. Pick a vertex  $B=i_1i_2\cdots i_n$ in $\mathcal{F}_{n}^m$ not containing $1$ such that $B\in \text{st}_{K_1}(D_1)\cap \text{st}_{K_1}(D_2)$ with $D_1, D_2\in \mathcal{G}_0$. There are three cases to discuss.
\begin{itemize}
\item[Case 1:] Suppose $|D_1|=|D_2|=n+2$. Then by Lemma~\ref{subset_dist2}, $B\subset D_1$ and $B\subset D_2$. Since both $D_1$ and $D_2$ contain $1$, $|D_1\cap D_2|=n+1$ and therefore $\{D_1, D_2\} \in K_0$.

\item[Case 2:] Suppose $|D_1|=n$ and $|D_2|=n+2$. Then $D_1$ contains an $(n-1)$-subset of $B$ and $1$; hence $D_1\subset D_2$. By Lemma~\ref{subset_dist2}, $d(D_1, D_2)=2$ and therefore $\{D_1, D_2\} \in K_0$.

\item[Case 3:] Suppose $|D_1|=|D_2|=n$. Then both $D_1$ and $D_2$ contains an $(n-1)$-subset of $B$ and $1$ and hence $|D_1\cap D_2|=n-1$, i.e.,  $d(D_1, D_2)=2$. Therefore $\{D_1, D_2\} \in K_0$.
\end{itemize}

Now fix $A\in \mathcal{F}^m_{n+2}$ with $\min A\geq 2$ and assume for induction that $\mathcal{VR}(\{B\in \mathcal{F}^m_{n}\cup \mathcal{F}^m_{n+2}: B\prec A \}, 2)$ is homotopy equivalent to
$$\mathcal{VR}(\mathcal{F}^{m-1}_{n-1}\cup \mathcal{F}^{m-1}_{n+1}, 2)\vee \bigvee_{\sum_{B\in\mathcal{F}^m_{n+2} \text{ with }\min B\geq 2 \text{ and } B\prec A} s_B}S^3$$
which is a wedge sum of $S^3$'s. If $A=\min_{\prec} \{C: C\in \mathcal{F}^m_{n+2}\text{ and }\min C=2\}$, then the set $\{B: B\in\mathcal{F}^m_{n+2} \text{ with }\min B\geq 2 \text{ and } B\prec A\}$ is empty; so the inductive assumption holds since $\mathcal{VR}(\{B\in \mathcal{F}^m_{n}\cup \mathcal{F}^m_{n+}: B\prec A \}, 2)\simeq \mathcal{VR}(\mathcal{F}^{m-1}_{n-1}\cup \mathcal{F}^{m-1}_{n+1}, 2)$ by the discussion above.

Let $L=\mathcal{VR}(\{B\in \mathcal{F}^m_{n+2}\cup \mathcal{F}^m_{n}: B\preceq A \}, 2)$. Then by Lemma~\ref{lk_wide_2}, $\text{lk}_L(A)$ is homotopy equivalent to $\bigvee_{s_A}S^2$ which is clearly contractible in $L\setminus\{A\}$. Hence by Lemma~\ref{complex_add_1v}, $L$ is homotopy equivalent to
$$\mathcal{VR}(\mathcal{F}^{m-1}_{n+1}\cup \mathcal{F}^{m-1}_{n+1}, 2)\vee \bigvee_{\sum_{B\in\mathcal{F}^m_{n+2} \text{ with }\min B\geq 2 \text{ and } B\prec A} s_B}S^3 \vee \Sigma (\bigvee_{s_A}S^2),$$
i.e.
$$\mathcal{VR}(\mathcal{F}^{m-1}_{n+1}\cup \mathcal{F}^{m-1}_{n+1}, 2)\vee \bigvee_{\sum_{B\in\mathcal{F}^m_{n+2} \text{ with }\min B\geq 2 \text{ and } B\preceq A} s_B}S^3.$$

This finishes the proof. \end{proof}

We conclude this section by showing that the vertices $\mathcal{F}_{n+1}^m$ in the complex $\mathcal{VR}(\mathcal{F}_n^{m}\cup\mathcal{F}_{n+1}^{m}\cup \mathcal{F}_{n+2}^m, 2)$ don't contribute to its homotopy type which means that it is homotopy equivalent to $\mathcal{VR}(\mathcal{F}_n^{m}\cup \mathcal{F}_{n+2}^m, 2)$.

\begin{theorem} Suppose that $1\leq n<m-3$ with $m\geq 4$. Then, $$\mathcal{VR}(\mathcal{F}_n^{m}\cup\mathcal{F}_{n+1}^{m}\cup \mathcal{F}_{n+2}^m, 2)\simeq\mathcal{VR}(\mathcal{F}_n^{m}\cup \mathcal{F}_{n+2}^m, 2).$$
\end{theorem}

\begin{proof} Let $K=\mathcal{VR}(\mathcal{F}_n^{m}\cup\mathcal{F}_{n+1}^{m}\cup \mathcal{F}_{n+2}^m, 2)$ and $K_0=\mathcal{VR}(\mathcal{F}_n^{m}\cup \mathcal{F}_{n+2}^m, 2)$. Then we claim that $K=\text{SC}_{K}(K_0)$ and $\text{SC}_{K}(K_0)\simeq K_0$.

It is clear that $\text{SC}_{K}(K_0)\subseteq K$. Take a simplex $\sigma$ in $K$ such that none of its vertices is in $K_0$; hence all its vertices are in $\mathcal{F}_{n+1}^{m}$. By Lemma~\ref{max_simp}, $\sigma$ is a face of either $N[i_1, i_2, \ldots, i_{n}]$ or $L[i_1, i_2, \ldots, i_{n+2}]$. Note that $\{i_1 i_2 \cdots i_{n}\}\cup N[i_1, i_2, \ldots, i_{n}]$ is a simplex in $K$ with $i_1 i_2 \cdots i_{n}\in K_0$; therefore $N[i_1, i_2, \ldots, i_{n}]\in \text{SC}_K(K_0)$. Also  $\{i_1i_2\cdots i_{n+2}\}\cup L[i_1, i_2, \ldots, i_{n+2}]$ is a simplex in $K$ with $i_1i_2\cdots i_{n+2}\in K_0$; hence, $L[i_1, i_2, \ldots, i_{n+2}] \in \text{SC}_K(K_0)$. Therefore,  $\text{SC}_{K}(K_0)= K$.

Take $D\in \mathcal{F}^m_{n+1}$ with $D\in \text{st}_K(B_1)\cap \text{st}_K(B_2)$ where $B_1, B_2$ are vertices in $K_0$. Using a similar discussion as in the proof of Theorem~\ref{homo_2}, $\{B_1, B_2\}\in K_0$. Hence the condition of Lemma~\ref{SC_homo} is satisfied which implies that  $\text{SC}_{K}(K_0)\simeq K_0$.

Therefore, we conclude that $K\simeq K_0$. \end{proof}

\section{Open Questions}

There is little known about the Vietoris-Rips complexes of these finite metric spaces with large scales. A good number of interesting open questions about the Vietoris-Rips complex on hypercube groups with large scales have  been raised in \cite{AA22, Shu22}. We'll end our paper with a couple questions related to the independence complex of Kneser graphs.

Suppose $2<n<m-2$. For any pair of subsets $B_1, B_2$ of $[m]$ with $|B_1|=|B_2|=n$, $d(B_1, B_2)\leq 2k+1$ is equivalent to $d(B_1, B_2)\leq 2k$ for any nonnegative integer $k$. Hence the Vietoris-Rips complex $\mathcal{VR}(\mathcal{F}_n^m, 3)$ is identical with $\mathcal{VR}(\mathcal{F}_n^m, 2)$. Little is known for larger scale $r\geq 4$. The complex $\mathcal{VR}(\mathcal{F}_3^6, 4)$ is  the boundary of a cross-polytope on $20$ vertices, hence it is homotopy equivalent to $S^9$. Using polymake \cite{Poly10}, we find the reduced homology groups of $\mathcal{VR}(\mathcal{F}_3^7, 4)$ is trivial when $n\neq 6$ or $9$; also,  $\tilde{H}_6( \mathcal{VR}(\mathcal{F}_3^7, 4)) =\mathbb{Z}^{29}$ and $\tilde{H}_9 (\mathcal{VR}(\mathcal{F}_3^7, 4)) =\mathbb{Z}^7$. This is related to independence complex of the Kneser graphs. Notice that the complex $\mathcal{VR}(\mathcal{F}_3^m, 4)$ is identical with $\mathcal{VR}(\mathcal{F}_3^m, 5)$; therefore both of them are equal to the independence complex of the Kneser graph $\text{KG}_{3, m-6}$ with $m\geq 6$.  Then the complex $\mathcal{VR}(\mathcal{F}_n^m, 4)$ for general $2n<m$ is very likely to be homotopy equivalent to a wedge sum of spheres with different dimensions.

Then, we have the following question.

\begin{ques}Assume that $2n<m$. Are the complexes $\mathcal{VR}(\mathcal{F}_n^m, 4)$ with $2n<m$ homotopy equivalent to a wedge sum of spheres $S^6$'s and $S^9$'s?
\end{ques}

In general, it is worth to investigate the following question.

\begin{ques} What are the homotopy types of the complex $\mathcal{VR}(\mathcal{F}^m_n, r)$ for $r\geq 4$?

\end{ques}

\textbf{Acknowledgements} The authors are grateful to Professor Henry Adams and the anonymous referees for their valuable comments and suggestions which lead to the improvements of the paper.

\end{document}